\documentclass[twoside,10.5pt]{article}
\usepackage{amsfonts}
\usepackage{mathrsfs}
\usepackage{pifont}
\usepackage{amsmath}
\usepackage{amsthm}
\usepackage{txfonts}
\usepackage{geometry}
\usepackage{latexsym}
\usepackage{amssymb}
\usepackage{graphicx}
\usepackage{geometry}
\usepackage{xcolor}

\newtheorem{definition}{Definition}

\usepackage{ytableau}
\usepackage[colorlinks=true, allcolors=blue]{hyperref}

\newtheorem{theorem}{Theorem}[section]

\newtheorem{proposition}{Proposition}[section]
\usepackage{fancyhdr}
\begin{document}

\title{\parbox{\linewidth}{\footnotesize\noindent } A Characterization of the Gelfand pair : Quasi Gelfand pair}
\author{Cornelie MITCHA MALANDA\thanks{{\scriptsize Brazzaville, Congo. E-mail: }{\footnotesize %
Cornelie.mitcha@yahoo.fr}}}
\date{}
\maketitle

\begin{abstract}
    Let G be a locally compact group and let K be a compact subgroup of Aut(G), the group of automorphisms of G. The pair $(G, K )$ is a Gelfand pair if the algebra $L^{1}_{K}(G)$ of K-invariant integrable functions on G is commutative under convolution. In \cite{toure2008lie}, the charactezations of this algebra in the nilpotent case were studied, which generalize some results obtained by C. Benson, J. Jenkins, G. Ratcliff in \cite{benson1990gel} and obtained a new criterion for Gelfand pairs. In this paper we describe the spherical function associated with this type of pair.
\end{abstract}



\section{Introduction}

Let  G be a locally compact group and K be a compact subgroup of automorphisms of G. We denote by $L_{K}^{1}(G)$ the algebra (under convolution) of K-invariant integrable function with complex values on G. We say (G,K) is a Gelfand pair when the algebra $L_{K}^{1}(G)$ is commutative. Equivalently, the algebra $L^{1}(G//K)$ of integrable K-biinvariant functions on the semidirect product $K \ltimes G$ of K is abelian. This notion of the Gelfand pair has been sufficiently studied by several authors (\cite{benson1990gel}, \cite{benson1999orbit}, \cite{faraut1980analyse}). However, with the Lie algebra structure induced by the convolution product, $L_{K}^{1}(G)$ is now a Lie algebra. In \cite{toure2008lie} some characterization of this algebra when it is nilpotent is given, which generalized some results obtained by C. Benson, J. Jenkins and G. Ratcliff \cite{benson1990gel}, and the new Gelfand pair criterion is obtained. In this work, we are interesred in this algebra when the nilpotence step is p ($p \in \mathbb{N}^{\ast}$). Hence, we say that (G, K) is a quasi-Gelfand pair when the algebra $L_{K}^{1}(G)$ is nilpotent of step p. When p is equal to 1 we recover the classical Gelfand pair with the zonal spherical function. Throughout this work, we describe the analogue of spherical function called the quasi-spherical function, which is associated with this type of pair and give an example of this pair..

\section{Quasi Gelfand pair}

Let $G$ be a locally compact group, $K$ a compact subgroup of $G$.
We designate by $L^{1}(G)^{\natural}$ the convolution algebra of $K$-biinvariant integrable functions. This algebra is endowed with the Lie algebra structure by puting $[f,g]=f\ast g-g\ast f$ for all functions $f$ and $g$ belonging to $L^{1}(G)^{\natural}$.
Every homomorphism $\pi$ of convolution algebra is an homomorphism of Lie algebra. So
\begin{eqnarray}
\pi (f\ast g)=\pi (f)\pi (g)\Longrightarrow \pi ([f,g])=[\pi (f),\pi (g)]
\end{eqnarray}%
for all functions $f,g\in L^{1}(G)^{\natural }$.

If K is a compact subgroup of Aut($G$), the group of automorphisms of $G$ and  $dk$ the normalized Haar measure on K, we put, for all function $f \in
L^{1}(G)$, 
$$f^{K}(x) =\int_{K}f(k.x)dk $$
for all $x \in G$. A function $f$ is
$K$-invariant if only if $f^{K} = f$.

$L_{K}^{1}\left( G\right) $ denotes the subalgebra of $L^{1}(G)$ of 
$K$-invariant functions, is isomorphic to $L^{1}\left( G\right) ^{\natural}$ (see \cite{faraut1980analyse} for details).

 \begin{definition}
     The couple $(G , K)$ is called quasi-Gelfand pair of order $p~(p \in \mathbb{N})$ if  $L^{1}(G)^{\natural}$ or $L_{K}^{1}\left( G\right) $   is nilpotent of step $p$.
 \end{definition}

We put:
\begin{eqnarray*}
C^{1}(L^{1}(G)^{\natural })
&=&[L^{1}(G)^{\natural },L^{1}(G)^{\natural
}]\\
&\vdots & \\
C^{p}(L^{1}(G)^{\natural })
&=& [L^{1}(G)^{\natural },C^{p-1}(L^{1}(G)^{\natural
})].
\end{eqnarray*}

Let consider $(G,K)$ a quasi Gelfand pair and $E$ a vector space of dimension $n$. For all $X\in L^{1}(G)^{\natural })$, one has: 
\begin{equation*}
adX[C^{p-1}(L^{1}(G)^{\natural })] = ad^{p}X((L^{1}(G)^{\natural })).
\end{equation*}

Thus any function $f\in C^{p}(L^{1}(G)^{\natural })$ can be written as $%
f=ad^{p}h(k)$, with $h,k\in L^{1}(G)^{\natural }$.

\begin{theorem}
If $(G , K)$ is a quasi-Gelfand pair then $G$ is unimodular.
\end{theorem}

\begin{proof}
Suppose $L^{1}(G)^{\natural}$ is the Lie algebra of step $p$ then $C^{(p-1)}(L^{1}(G)^{\natural}) \neq \{0\}$ and $C^{(p)}(L^{1}(G)^{\natural}) = \{0\}$, where $C^{(p)}(L^{1}(G)^{\natural}) = [ L^{1}(G)^{\natural} , C^{(p-1)}(L^{1}(G)^{\natural}].$
Thus, for all functions $f \in L^{1}(G)^{\natural}$ and $g \in C^{(p-1)}(L^{1}(G)^{\natural}))$ not identically equal to zero, one has  
$f*g = g*f$ and for all $x \in G$,

\begin{eqnarray*}
\int_{G} f(y) g(y^{-1}x)dy
&=& \int_{G} \Delta(y^{-1}) f(y)g(xy^{-1})dy 
\end{eqnarray*}
hence, 
\begin{eqnarray*}
\int_{G} f(y) [g(y^{-1}x) - \Delta(y^{-1}) g(xy^{-1})] dy
&=& 0.
\end{eqnarray*}

Thus, for any function $f \in L^{1}(G)^{\natural}$  

\begin{equation*}
0 = \int_{G} f(y) \int_{K} [g(y^{-1}(kx)) - \Delta(y^{-1}) g((kx)y^{-1})] dk dy
\end{equation*}

which implies

\begin{eqnarray*}
\int_{K} [g(y^{-1}(kx)) - \Delta(y^{-1}) g((kx)y^{-1})] dk 
&=& 0
\end{eqnarray*}
for all $y \in G$ and $g \in C^{(p-1)}(L^{1}(G)^{\natural})$.\\
In particular, for $x = e$ we obtain $g(y^{-1}) - \Delta(y^{-1}) g(y^{-1}) = 0 ~~for~all ~x \in G$. We deduce that\\
 $\Delta(y) = 1$ for all $y \in G$ thus $G$ is unimodular.
\end{proof}

Consider now G a Lie group, we denote Then we designate by $\xi^{'}(G)$ the space of compactly supported distributions and equip the space  with the following bracket.
\begin{align}
    [D_{1}, D_{2}] = D_{1}*D_{2} - D_{2} \ast D_{1}
\end{align}
for all $D_{1}, D_{2} \in \xi^{\prime}(G)$. Then the convolution of distributions $D_{1}, D_{2} \in \xi^{\prime}(G)$ is defined by $< D_{1} \ast D_{2},f> = <D_1, \varphi>$ where $\varphi(x)=<D_2 , _{x^{-1}} f>$ for all $x \in G$ and ${ _{x^{-1}}} f(y) = f(xy)$ for all $y \in G$.
K is always a compact subgroup of Aut(G). Let $\xi_{K}^{\prime} (G)$ denote the algebra pf K-invariant compactly supported distributions that is the distribution $D \in \xi^{\prime} (G)$ such that $D^{K}= D$ where $D^{K}$ is defined by
$< D^{K} , f > = < D , f^{K} >$ for all $f \in \mathbf{C}_{c}^{\infty}(G)$. $\xi_{K}^{\prime}(G)$ is a Lie subalgebra of $\xi^{\prime}(G)$.

Let $\delta_{x}$ be the Dirac measure at $x \in G$ and for any function $f \in \mathbf{C}_{c}(G)$, the space of continuous complex functions on $ G$ compact support,

\begin{equation*}
<\delta_{x}^{K} , f> = <\delta_{x} , f^{K}> = f^{K}(x) = \int_{K} f(k.x)dk,
\end{equation*} 

$\delta_{x}^{K}$ is a $K$-invariant distribution with compact support (the support being $K.x$). Let $x_{1}.x_{2}.\cdots.x_{p}$ be elements of $G$, we have

\begin{equation*}
<\delta_{x_{1}}^{K}*\delta_{x_{1}}^{K}*\cdots*\delta_{x_{p}}^{K} , f> = \int_{K^{p}} f((k_{1}.x_{1})(k_{2}.x_{2})\cdots(k_{p}.x_{p}))dk_{1}dk_{2}\cdots dk_{p}
\end{equation*}

for any function $f \in \mathbf{C}_{c}(G)$. Let $\mathcal{S}_{p}$ denote the set of permutations of order $p$. For $i \in \{1,2,\cdots, p\}$, 

\begin{eqnarray*}
\mathcal{S}_{p}^{i} 
&=& \{\sigma \in \mathcal{S}_{p}: \sigma(i) = p; \sigma(1)<\sigma(2)<\cdots < \sigma(i)~~\rm{and}\\
& &
\sigma(i+1) > \sigma(i+2)>\cdots > \sigma(p) \}.
\end{eqnarray*}

Note that

\begin{eqnarray*}
\mathcal{S}_{p}^{1} 
&=& \{\sigma \}~~where ~~\sigma(1) = p, \sigma(2) = p-1, \cdots, \sigma(p)=1\\
\mathcal{S}_{p}^{p}
&=& \{id \}.
\end{eqnarray*}

Generally $\mathcal{S}_{p}^{i} = \mathcal{S}_{p}^{i,1} \cup \mathcal{S}_{p}^{i,1^{\sim}}$ where

\begin{equation*}
\mathcal{S}_{p}^{i,1} = \{ \sigma \in \mathcal{S}_{p}^{i}; \sigma(1)=1\} 
\end{equation*}

and 

\begin{equation*}
\mathcal{S}_{p}^{i,1^{\sim}}  = \{ \sigma \in \mathcal{S}_{p}^{i}; \sigma(p)=1\}.
\end{equation*}

The following result is a generalization of [proposition 1.4.2.6, \cite{benson1990gel}].

\begin{theorem}
    
If $(G, K)$ is a quasi- Gelfand pair of order $p$ then for all $x_{1}, x_{2}, \cdots, x_{p+1} \in G$ there exists a $m \in \{1, 2, \cdots, p\}$ and a permutation $\sigma \in \mathcal{S}_{p+1}^{m}$ such that $x_{1}x_{2}x_{3}\cdots x_{p+1} \in (K.x_{\sigma(1)})(K.x_{\sigma(2)})\cdots (K.x_{\sigma(p+1)}).$
\end{theorem}
\begin{proof}
    See \cite{toure2008lie}.
\end{proof}

\section{Quasi Spherical function}

The quasi-spherical function is an analog of spherical function (SF), the general theory of SF can be found in \cite{faraut1980analyse}.

\begin{definition}
    The quasi-spherical function $\varphi $ is a continuous function on $G$ with values in the space of endomorphisms of $E$ such that the map
\begin{equation*}
\pi _{\varphi }:L^{1}(G)^{\natural }\longrightarrow End(E),f\longmapsto \pi
_{\varphi }(f)=\int_{G}f(x)\varphi (x^{-1})dx
\end{equation*}
verifies the following relation:
\begin{equation*}
\pi _{\varphi }(f\ast g)=\pi _{\varphi }(f)\pi _{\varphi }(g),~~for~all
~~f\in L^{1}(G)^{\natural }~~and~~g\in C^{p-1}(L^{1}(G)^{\natural }).
\end{equation*}
\end{definition}

\begin{proposition}
Let $\varphi $ be a continuous function on $G$, biinvariant by $K$,
not identically zero.

\begin{enumerate}
\item The function $\varphi $ is quasi-spherical if and only if $x,y\in G$,
\begin{equation*}
\int_{K}\varphi (xky)dk=\varphi (y)\varphi (x)
\end{equation*}%
where $dk$ is a normalized Haar measure of the compact subgroup $K$.

\item The function $\varphi $ is quasi-spherical if and only if

\begin{enumerate}
\item $\varphi (e)= Id_{E}$

\item For all function $g$ of $L^{1}(G)^{\natural }$, there exists an endomorphism $\pi
_{\varphi }(g)$
such that :%
\begin{equation*}
g\ast \varphi =\varphi \pi
_{\varphi } (g)\text{.}
\end{equation*}
\end{enumerate}
\end{enumerate}
\end{proposition}

\begin{proof}

\begin{enumerate}

For all $g\in C^{p-1}(L^{1}(G)^{\natural })$ and $f\ast
g=\int_{G}f(y)g(y^{-1}x)dy$, we have 
\begin{eqnarray*}
\pi
_{\varphi } (f\ast g)
&=&\int_{G}f(y)\int_{G}g(x)\varphi (x^{-1}y^{-1})dxdy \\
&=&\int_{G}f(y)\int_{G}ad^{p-1}h(k)(x)\varphi (x^{-1}y^{-1})dxdy
\end{eqnarray*}%
Via $x\longrightarrow kx$, then%
\begin{eqnarray*}
\pi
_{\varphi } (f\ast g)
&=&\int_{G}\int_{G}f(y)ad^{p-1}h(k)(x)(\int_{K}\varphi
(x^{-1}k^{-1}y^{-1})dk)dxdy
\end{eqnarray*}%
and
\begin{eqnarray*}
\pi
_{\varphi } (f\ast g)-\pi
_{\varphi } (f)\pi
_{\varphi } (g)\pi
_{\varphi } (g)
&=&\int_{G}\int_{G}f(y)ad^{p-1}h(k)(x)(\int_{K}\varphi
(x^{-1}k^{-1}y^{-1})dk)dxdy \\
&-&(\int_{G}f(y)\varphi (y^{-1})dy)(\int_{G}ad^{p-1}h(k)(x)\varphi
(x^{-1})dx) \\
&=&\int_{G}\int_{G}f(y)ad^{p-1}h(k)(x)(\int_{K}\varphi
(x^{-1}k^{-1}y^{-1})dk-\varphi (y^{-1})\varphi (x^{-1}))dxdy
\end{eqnarray*}%
If $\varphi $ is a quasi-spherical function then 
\begin{equation*}
\pi
_{\varphi } (f\ast g)-\pi
_{\varphi } (f)\pi
_{\varphi } (g)=0\text{.}
\end{equation*}%
Thus,%
\begin{equation*}
\int_{K}\varphi (x^{-1}k^{-1}y^{-1})dk-\varphi (y^{-1})\varphi (x^{-1})=0.
\end{equation*}%
This implies that 
\begin{equation*}
\int_{K}\varphi (x^{-1}k^{-1}y^{-1})dk=\varphi (y^{-1})\varphi (x^{-1})~%
\text{~locally~~almost~~everywhere.}
\end{equation*}%
So $\int_{G}\varphi (xky)dk=\varphi (y)\varphi (x)$ locally almost everywhere.

Conversely, if%
\begin{equation*}
\int_{K}\varphi (xky)dk=\varphi (y)\varphi (x)\text{,}
\end{equation*}%
then $\pi
_{\varphi } (f\ast g)=\pi
_{\varphi } (f)\pi
_{\varphi } (g)$ for all $f\in L^{1}(G)^{\natural
},~g\in C^{p-1}(L^{1}(G)^{\natural })$ and $\varphi $ a quasi-spherical function.

\item Suppose $\varphi $ is a quasi-spherical function on $G$.

\begin{enumerate}
\item $\varphi (e)=Id_{E}$, indeed

\begin{equation*}
\varphi (x)=\varphi (ex)=\int_{K}\varphi (ekx)dk=\varphi (x)\varphi
(e)
\end{equation*}

 for all $x\in G$, therefore $\varphi (e)=Id_{E}.$

\item 
\begin{eqnarray*}
g\ast \varphi (x)
&=&\int_{G}g(y)(\int_{K}\varphi (y^{-1}kx)dk)dy \\
&=&\varphi (x)\int_{G}g(y)\varphi (y^{-1})dy \\
&=&\varphi (x)\pi
_{\varphi } (g)~~\forall x\in G,
\end{eqnarray*}

hence $g\ast \varphi =\varphi \pi
_{\varphi } (g).$

Conversely, for all $f\in L^{1}(G)^{\natural }$ and $g\in C^{p-1}(L^{1}(G)^{\natural })$

\begin{eqnarray*}
\pi
_{\varphi } (f\ast g) 
&=&\int_{G}\int_{G}f(y)g(x)\varphi (x^{-1}y^{-1})dxdy \\
&=&\int_{G}f(y)g\ast \varphi (y^{-1})dy\\
&=&\int_{G}f(y)\varphi (y^{-1})\pi (g)dy\\
&=&\pi
_{\varphi } (f)\pi
_{\varphi } (g)~~\forall f\in
L^{1}(G)^{\natural }~et~g\in C^{p-1}(L^{1}(G)^{\natural }).
\end{eqnarray*}%
Therefore $\varphi $ is a quasi-spherical function.
\end{enumerate}
\end{enumerate}
\end{proof}

\begin{theorem}
Let $\varphi$ be a bounded quasi-spherical function, the application 
\begin{equation*}
f \longmapsto \pi
_{\varphi }(f) = \int_{G} f(x) \varphi(x^{-1}) dx
\end{equation*}
is the representation of  convolution algebra $%
L^{1}(G)^{\natural}$, and all representation not identically zero of $%
L^{1}(G)^{\natural}$ is of this form.
\end{theorem}

\begin{proof}

Let $f \in L^{1}(G)^{\natural}$ and $\mu$ a Haar measure on $G$. Consider the function $\varphi_{f}$ defining by 
\begin{equation*}
\varphi_{f}(g) = \int_{G} f(x) g(x) d\mu(x) ~~for~ all ~g \in
~C^{p-1}(L^{1}(G)^{\natural}).
\end{equation*}

Since $\varphi _{f}$ is continuous, there exists a function $%
f_{0}\in L^{1}(G)$ such that%
\begin{equation*}
\varphi _{f}(g)=\int_{G}gf_{0}d\mu (x).
\end{equation*}

Let $f_{0}\in L^{1}(G)^{\natural }$. For all function $g\in
C^{p-1}(L^{1}(G)^{\natural })$ we have :%
\begin{eqnarray*}
\pi
_{\varphi } (g)=\pi
_{\varphi } (g\ast f_{0}) &=&(\int_{G}\int_{G}g(y)f(y^{-1}x)\varphi
_{0}(x^{-1})d\mu (x)d\mu (y)) \\
&=&\int_{G}g(y)\varphi (y^{-1})d\mu (y)
\end{eqnarray*}%
where one put $\varphi (y^{-1})=\int_{G}f_{0}(y^{-1}x)\varphi
_{0}(x^{-1})d\mu (x)$. The function is $K$-biinvariant and%
\begin{equation*}
\varphi (e)=\int_{G}f_{0}(ex)\varphi _{0}(x^{-1})d\mu (x)\pi
(f_{0})^{-1}=\int_{G}f(x)\varphi (x^{-1})dx=f\ast \varphi (e)=\varphi (e)\pi
_{\varphi }
(f)
\end{equation*}%
which implies $\varphi (e)=Id_{E}$ for all functions $f$,%
\begin{eqnarray*}
\pi
_{\varphi } (f\ast g)
&=&\int_{G}\int_{G}f(y)g(x)\varphi (x^{-1}y^{-1})d\mu (x)d\mu (y) \\
&=&\int_{G}f(y)g\ast \varphi (y^{-1})d\mu (y)
\end{eqnarray*}%
\begin{equation*}
\pi
_{\varphi } (f)\pi
_{\varphi } (g)=(\int_{G}f(y)\varphi (y^{-1})d\mu (y))\pi
_{\varphi } (g)
\end{equation*}

which implies 
\begin{eqnarray*}
\pi
_{\varphi }(f*g) -\pi
_{\varphi }(f)\pi
_{\varphi }(g)= \int_{G} f(y) (g*\varphi(y^{-1}) -
\varphi(y^{-1})\pi
_{\varphi }(g)) d\mu(y) = 0
\end{eqnarray*}
this is true for all functions $f \in L^{1}(G)^{\natural}$ and $g
\in C^{p-1}(L^{1}(G)^{\natural})$, hence 
\begin{equation*}
g*\varphi(y) = \varphi(y)\pi
_{\varphi }(g)
\end{equation*}
almost everywhere, then $\varphi$ is almost everywhere equal to quasi-spherical function.
\end{proof}

\paragraph{Comment on the Gelfand transform associted:} 
For all $f\in L^{1}(G)^{\natural }$, we call general Gelfand transform of $f$, the map $g_{f}$ of $X^{n}(L^{1}(G)^{\natural
})$ in $M(n,\mathbb{C})$ defining by, 
\begin{eqnarray*}
g_{f}:X^{n}(L^{1}(G)^{\natural }) &\longrightarrow &M(n,\mathbb{C}) \\
\pi &\longmapsto & \pi (f),
\end{eqnarray*}%
where $X^{n}(L^{1}(G)^{\natural })$ is the representation space of convolution algebra $L ^{1}(G)^{\natural }$ of dimension $n$, the map $f\longrightarrow g_{f}$ of $L^{1}(G)^{\natural }$ in $M(n,%
\mathbb{C})$ is called \textbf{Gelfand transform} associated to $L^{1}(G)^{\natural }$  \cite{AMBP_1996__3_2_117_0}.

Let $(G , K)$ be a quasi-Gelfand pair, $S(G/K)$ the quasi-spherical bounded function space on $G$ with values in $M(n , \mathbb{C}%
)$ and $\mu$ a Haar measure on $G$. If $\varphi \in S(G/K) $ then
the map $f \longmapsto \pi_{\varphi}(f)$ is a representation of convolution algebra $L^{1}(G)^{\natural}$ and all representations of $L ^{1}(G)^{\natural}$ is of this form. Thus, the
map $\varphi \longrightarrow \pi_{\varphi}$ is bijective and allows to identify
 the space $S(G/K)$ to the representations space
with dimension $n$ of the convolution algebra $L^{1}(G)^{\natural}$. 
\newline
Therefore, there exists the map $\mathcal{F}_{f}$ of $S(G/K)$
in $M(n , \mathbb{C})$ such that: 
\begin{equation*}
\mathcal{F}f(\varphi) = \int_{G} f(x) \varphi(x^{-1}) d\mu(x).
\end{equation*}

Let $f\in L^{1}(G)^{\natural }$. We call the quasi-spherical Fourier transform of the function $f$, the map $\mathcal{F}f$ defining on $S(G/K)$ by
\begin{equation*}
\mathcal{F}f(\varphi )=\int_{G}f(x)\varphi (x^{-1})d\mu (x).
\end{equation*}

We call the quasi-spherical Fourier cotransform of the function $f
$, the map $\overline{\mathcal{F}}f$ define by : 
\begin{equation*}
(\overline{\mathcal{F}}f)(\varphi )=\int_{G}f(x)\varphi (x)d\mu (x).
\end{equation*}

\section{Examples}

\begin{enumerate}
    \item 

  Let $H_{1}$ be the 3-dimensional Heisenberg group whose the coordinates are $(x, y, t) \in \mathbb{R}^{3}$ and consider $N =\mathbb{C}\times H_{1}, K=U(1)$ and K acts on N by 

    \begin{align*}
        k.(w_{1}, w_{2}, t) = (kw_{1}, kw_{2}, t)~~\forall w_{1}, w_{2} \in \mathbb{C}, t \in \mathbb{R}.
    \end{align*}

    $\eta^{*}$ has a base $\{ \alpha_{1}, \alpha_{2}, \beta_{1}, \beta, \lambda \}$ where 

 $\alpha_{1}(w_{1}, w_{2}, t) = Re(w_{1}), ~\alpha_{2}(w_{1}),w_{2}, t) = Im(w_{1}), ~\beta_{1}(w_{1}, w_{2}, t) = Re(w_{2}), ~\beta_{2}(w_{1}, w_{2}, t)$ and $\lambda(w_{1}, w_{2}, t)=t. $ 

Let $l=\alpha_{1} + \lambda$. Then 

$$\mathcal{O}_{l}^{N}= l + span(\beta_{1}, \beta_{2}).$$
 Since 
 $$k.l(w_{1}, w_{2}, t) = Re(w_{1/k}) + t,$$
 it is clear that $k.l ~\rm{in} ~\mathcal{O}_{l}^{N}$ if and only if $k=1$. Hence, $K_{l}=\{1\}$ and its Lie algebra $=\{0\}$. The following scalar product 

 $$\langle (x+iy, u+iv, t), (x^{\prime}+iy^{\prime}, u^{\prime}+iv^{\prime}, t) \rangle = xx^{\prime}+yy^{\prime} + uu^{\prime} + vv^{\prime} + tt^{\prime}$$
allow us to obtain the below decomposition

\begin{align*}
    z = \mathbb{C} \times \{0\}\times \mathbb{R},~~z_{l} = span((i,0,1),(1,0,-1)),~~z_{l}^{\prime} = span((1,0,1)), ~~a_{l}=\{0\}~~b_{l}= m = span((0,1,0), (0,i,0))
\end{align*}

$H_{l}$ is the 3-dimension Heisenberg group of Lie algebra

$$h_{l}= b_{l}\oplus z^{\prime}_{l} = \{ (t,w,t): ~~t \in \mathbb{R}, w \in \mathbb{C}\}.$$

The bracket on $h_{l}$ is defined by 

\begin{align*}
    [(t, w, t), (t^{\prime}, w^{\prime}, t^{\prime})]_{l}= (-\frac{1}{2}Im(w\Bar{w^{\prime}}), 0, -\frac{1}{2}Im(w\Bar{w^{\prime}}))
\end{align*}
and the map 

\begin{eqnarray*}
    h_{l}
    &\longrightarrow & h_{1}\\
    (t, w, t)
    & \longmapsto & (\frac{w}{\sqrt{2}}, t)
\end{eqnarray*}

is a Lie algebra isomorphism. Hence $L_{K_{l}}^{1}(H_{l}) = L_{\{1\}}^{1} \cong L^{1}(H_{1})$ is not commutative. It follows from [Lemma 2.4 \cite{benson1999orbit}] that $(\mathbb{C}\times H_{1}, U(1))$ is not a Gelfand pair, and since $L_{U(1)}^{1}(\mathbb{C}\times H_{1})$ is nilpotent thus  $(\mathbb{C}\times H_{1}, U(1))$ is a nontrivial quasi-Gelfand pair.

\item  Let now consider the $H_{n}$ the Heisenberg group of dimension 2n+1. The pair $({H}_{n}, SO(n))$ is a trivial quasi-Gelfand pair.

Indeed, the action of $K = SO(n)\times T$ on $\mathbb{C}^{n}$ for $n \geq 3$,

$$(A,c)z = cAz$$ for $A \in SO(n,\mathbb{R}), c \in T, z \in \mathbb{C}^{n}$.

The decomposition of $\mathbb{C}[\mathbb{C}^{n}]$ under the action of $K$ is determined by the classical theory of spherical harmonics. The polynomial

$$\varepsilon(z) = z_{1}^{2}, \cdots ,z_{n}^{2} $$

is invariant under the action of $SO(n,\mathbb{R})$ on $\mathbb{C}[\mathbb{C}^{n}]$, i.e. let $k= (A,c) \in K$, then

\begin{eqnarray*}
(A,c)\varepsilon(z) 
&=& cA\varepsilon(z)\\
&=& c(Az_{1}^{2} + \cdots + A z_{n}^{2}).
\end{eqnarray*}

Hence $\varepsilon(z) \in \mathbb{C}[\mathbb{C}^{n}]$, thus $\varepsilon(z)$ is invariant under $K$. Consider

$$\Delta = \varepsilon(\partial_{z}) =\left(\frac{\partial}{\partial_{z_{1}}}\right) ^{2} +  \cdots + \left(\frac{\partial}{\partial_{z_{n}}}\right) ^{2}$$

and we define the space of harmonic polynomials by

$\mathcal{H} = \ker(\Delta: \mathbb{C}[\mathbb{C}^{n}] \rightarrow \mathbb{C}[\mathbb{C}^{n}])$. Thereby, 

$$\mathcal{H}_{m} = \mathbb{C}[\mathbb{C}^{n}]_{m} \cap \mathcal{H}$$

is the space of homogeneous harmonic polynomials of degree $m$ which are $SO(n,\mathbb{R})$-invariant by the invariance of $\Delta$.

Indeed, $\mathcal{H}_{m}$ is $SO(n,\mathbb{R})$-irreducible and

$$\mathbb{C}[\mathbb{C}^{n}]_{m} =  \mathcal{H}_{m} \oplus \varepsilon \mathbb{C}[\mathbb{C}^{n}]_{m-2}. $$

For nonzero integers $k, l,$

$$\mathbb{C}[\mathbb{C}^{n}]_{k,l} = m_{\varepsilon}^{l}\mathcal{H}_{k} = \varepsilon^{l}\mathcal{H}_{k}$$ 

is an irreducible $K$-invariant subspace of $\mathbb{C}[\mathbb{C}^{n}]_{k+2l}$.

$\mathbb{C}[\mathbb{C}^{n}]_{k,l}$ has dimension $h_{k}= \dim(\mathcal{H}_{k})$ independently of $l$.
Explicitly for $h_{0} = 1, h_{n}= n$ and

$$h_{k}= \dim(\mathbb{C}[\mathbb{C}^{n}]_{k}) - \dim(\mathbb{C}[\mathbb{C}^{n}]_{k-2}) = \begin{pmatrix} 
k+n-1 \\
k 
\end{pmatrix} \quad + \quad \begin{pmatrix} 
k+n-3  \\
k-2 
\end{pmatrix}$$ 

for $k \geq 2$.

The action of $K$ is multiplicity free because the elements $\mathbb{C}[\mathbb{C}^{n}]_{k,l}$ are pairwise equivalent.
 Indeed, $\mathbb{C}[\mathbb{C}^{n}]_{k,l}$ and $\mathbb{C}[\mathbb{C}^{n}]_{k^{ '},l^{'}}$ are of equal dimension when $k=k^{'}$ and the circle $T$ acts on this space by the elements $c \rightarrow c^{-( 2k+l)}$ and $c \rightarrow c^{-(2k^{'}+l^{'})}$ respectively. So we have following decomposition for this example
 \begin{equation*}
\mathbb{C}[\mathbb{C}^{n}] = \sum_{k,l} \mathbb{C}[\mathbb{C}^{n}]_{(k,l) \in \Lambda}.
\end{equation*}
Consequently, we deduce that the pair $({H}_{n}, SO(n,\mathbb{R}))$ is not a Gelfand pair because $\mathbb{ C}[\mathbb{C}^{n}]_{k,l}$ and $\mathbb{C}[\mathbb{C}^{n}]_{k^{'},l^{' }}$ are equivalent when $k=k^{'}$.

\end{enumerate}

\paragraph{Comment:} It turns out that when you consider the free two-step nilpotent Lie groups $\mathcal{N}_{v,2}$ and the special orthogonal group $SO(v)$, the pair $(\mathcal{N}_{v,2}, SO(v))$ is a Gelfand pair (see \cite{benson1990gel}, \cite{fischer2010bounded} for details).

\section{Acknowledgements}
I am grateful to professor Kinvi Kangni for introduce me to this topic.
\bibliographystyle{alpha}
\bibliography{Article}

\begin{thebibliography}{BJR99}

\bibitem[BJR90]{benson1990gel}
Chal Benson, Joe Jenkins, and Gail Ratcliff.
\newblock On gelfand pairs associated with solvable lie groups.
\newblock {\em Transactions of the American Mathematical Society},
  321(1):85--116, 1990.

\bibitem[BJR99]{benson1999orbit}
Chal Benson, Joe Jenkins, and Gail Ratcliff.
\newblock The orbit method and gelfand pairs, associated with nilpotent lie
  groups.
\newblock {\em The Journal of Geometric Analysis}, 9:569--582, 1999.

\bibitem[Far80]{faraut1980analyse}
Jacques Faraut.
\newblock Analyse harmonique sur les paires de gelfand.
\newblock {\em Les cours du CIMPA}, 1980.

\bibitem[Fis10]{fischer2010bounded}
Veronique Fischer.
\newblock The bounded spherical functions for the free two step nilpotent lie
  group, 2010.

\bibitem[KT96]{AMBP_1996__3_2_117_0}
Kinvi Kangni and Saliou Toure.
\newblock Transformation de {Fourier} sph\'erique de type $\delta $.
\newblock {\em Annales math\'ematiques Blaise Pascal}, 3(2):117--133, 1996.

\bibitem[TK08]{toure2008lie}
Ibrahima Toure and Kinvi Kangni.
\newblock On lie algebras of k-invariant functions.
\newblock {\em Journal of Mathematics of Kyoto University}, 48(4):847--855,
  2008.

\end{thebibliography}
\end{document}